\documentclass[12pt,a4paper,notitlepage,dvips]{article}
\usepackage[dvips]{graphicx}   
\usepackage[T1]{fontenc}  
\usepackage{times}     
\usepackage{enumerate,amsmath,amsthm,mathrsfs,citesort}
\usepackage{amsfonts}
\usepackage{amssymb}
\usepackage{enumitem}
\usepackage[a4paper,left=2.5cm,right=2cm,top=2.5cm,bottom=2.5cm]{geometry}
\usepackage{picins}
\usepackage{amstext} 
\usepackage{array}   
\setlist[1]{itemsep=0mm} 

\newtheoremstyle{bthm}{\baselineskip}{\baselineskip}{\slshape}{}{\bfseries}{}{ }{}
\newtheoremstyle{bex}{\baselineskip}{\baselineskip}{}{}{\sffamily}{:}{\newline }{}
\theoremstyle{bthm}
\newtheorem{thm}{Theorem}[section]
\newtheorem{cor}[thm]{Corollary}
\newtheorem{lem}[thm]{Lemma}

\newtheorem{prob}[thm]{Problem}

\theoremstyle{bex}

\begin{document}
\begin{titlepage}
\title{Bounds for the b-chromatic number of powers of hypercubes}
\author{P. Francis$^{1}$, S. Francis Raj$^{2}$, M. Gokulnath$^3$}
\date{{\footnotesize$^{1}$Department of Computer Science, Indian Institute of Technology, Palakkad-678557, India.\\
$^{2,3}$Department of Mathematics, Pondicherry University, Puducherry-605014, India.}\\
{\footnotesize$^{1}$: pfrancis@iitpkd.ac.in\, $^{2}$: francisraj\_s@yahoo.com\, $^3$: gokulnath.math@gmail.com }}
\maketitle
\renewcommand{\baselinestretch}{1.3}\normalsize
\begin{abstract}
The b-chromatic number $b(G)$ of a graph $G$ is the maximum $k$ for which $G$ has a proper vertex coloring using $k$ colors such that each color class contains at least one vertex adjacent to a vertex of every other color class.
In this paper, we mainly investigate on one of the open problems given in \cite{francis2017b}.
As a consequence, we have obtained an upper bound for the b-chromatic number of some powers of hypercubes.
This turns out to be an improvement of the already existing bound in \cite{francis2017b}.  
Further, we have determined a lower bound for the b-chromatic number of some powers of the Hamming graph, a generalization of the hypercube.
\end{abstract}

\section{Introduction}\label{intro}

All graphs considered in this paper are simple, finite and undirected.
Let $G$ be a graph with vertex set $V(G)$ and edge set $E(G)$.
A b-coloring of a graph $G$ using $k$ colors is a proper coloring of the vertices of $G$ using $k$ colors in which each color class has a color dominating vertex, that is, a vertex which has a neighbor in each of the other color classes.
The b-chromatic number, $b(G)$ of $G$ is the largest $k$ such that $G$ has a b-coloring using $k$ colors.
The concept of b-coloring was introduced by Irving and Manlove \cite{irving1999b} in analogy to the achromatic number of a graph $G$ (the maximum number of color classes in a complete coloring of $G$).
Since the time of its introduction, the concept of b-coloring has been extensively studied.  
Some of the references are \cite{balakrishnan2013bounds,barth2007b,blidia2012b,ibiapina2020b,maffray2013b,faik2004b,kouider2002some,kratochvil2002b}. Also recently, Jakovac and Peterin \cite{jakovac2018b} have published a survey article on b-coloring of graphs.

It is clear from the definition of $b(G)$ that the chromatic number, $\chi(G)$ of $G$ is the least $k$ for which $G$ admits a b-coloring using $k$ colors and hence $\chi(G) \leq b(G)\leq \Delta(G)+1$, where $\Delta(G)$ is the maximum degree of $G$.


For $q\geq2$, let $\mathbb{Z}_q=\{0,1,\ldots,q-1\}$ be the additive group of integer modulo $q$ and $\mathbb{Z}_q^n\cong \mathbb{Z}_q\times \mathbb{Z}_q\times\cdots\times \mathbb{Z}_q$ ($n$ times). 
The Hamming graph, $H_{n,q}$, has the vertex set $\mathbb{Z}_q^n$ and two $q$-ary $n$-tuples are adjacent if and only if they differ in exactly one coordinate.
The special case $H_{n,2}$ is the well known hypercube of dimension $n$ (generally denoted by $Q_n$).  
The hamming graphs have been investigated in the past.
See for instance \cite{ostergard1997coloring,kim2000coloring,jamison2008acyclic,guo2018b,harney2016robust}. 
The $p^{th}$ power of a graph $G$ denoted by $G^p$ is a graph whose vertex set $V(G^p)=V(G)$ and edge set $E(G^p)=\{xy : d_G(x,y)\leq p\}$, where $d_G(x,y)$ denotes the distance between $x$ and $y$ in $G$.

\begin{figure}[t]
	\centering
	\includegraphics[width=.5500\textwidth]{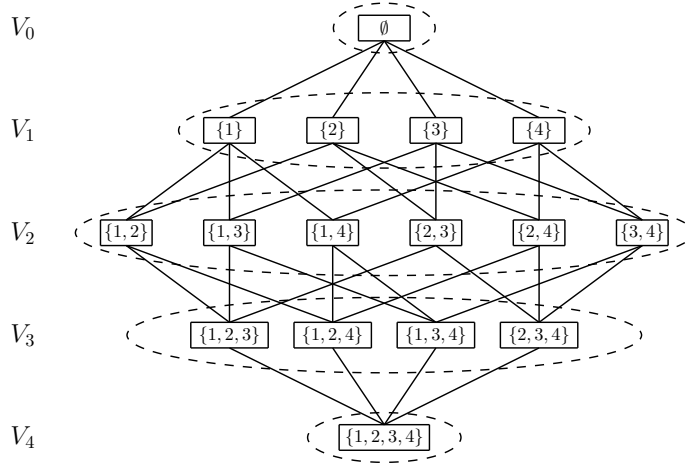}
	\caption{$Q_4$ where the boxes represent the vertices which are inside}
	\label{q4}
\end{figure}

For $n\in\mathbb{N}$, let $[n]=\{1,2,\ldots,n\}$ and  $2^{[n]}=\{x:x\subseteq[n]\}$, the power set of $[n]$.
Let us define the hypercube $Q_n$ in a slightly different way. 
The hypercube $Q_n$ is a graph with $V(Q_n)=2^{[n]}$ and for $x,y\in V(Q_n)$, $xy\in E(Q_n)$ if and only if $|x\triangle y|=1$. 
Throughout this paper, let us assume that the entries in $V(Q_n)=2^{[n]}$ are arranged in the ascending order. 
Let us define the simplicial ordering of the vertices of $Q_n$.
For two distinct vertices $x,y\in V(Q_n)$, we say that $x$ precedes $y$, denoted by $x < y$, if either $|x| < |y|$ or $|x| = |y|$ and $\min\{x\triangle y\} \in x$. The vertices in the hypercube are always taken in the simplicial ordering. 
For $0\leq i\leq n$, let $V_i=\{x\in2^{[n]}:|x|=i\}$.  
For better understanding, we have given $Q_4$ and $Q_5$ in Figure \ref{q4} and Figure \ref{q5} respectively. 
Let $\mathcal{I}_m$ denote the first $m$ elements of $2^{[n]}$ in the simplicial ordering.
Sometimes, we refer to $\mathcal{I}_m$ as an initial segment of size $m$ in the hypercube.
Also, let us define the simplicial ordering of sets consisting of vertices of $Q_n$ in the following way.
For two different sets $A,B\subseteq V(Q_n)$, we say that $A$ precedes $B$, denoted by $A < B$, if $|A| < |B|$ or $|A| = |B|$ and $\min\{A\triangle B\} \in A$, where $\min\{A\triangle B\}$ is the first vertex in the simplicial ordering of $A\triangle B$.
For $A\subseteq 2^{[n]}$ and $p\in[n]$, let us define $C^p[A]=\{y\in2^{[n]}:|x\triangle y|\leq p$ for every $x\in A\}$ and let $C^p(A)=C^p[A]\backslash A$.

Let us recall some of the definitions due to Tsukerman \cite{tsukerman2013isoperimetric} which are required for this paper.
For $A\subseteq2^{[n]}$ and $i\in[n]$, the $i$-sections of $A$ are given by $A_{i-}=\{x\in2^{[n]\backslash\{i\}}:x\in A\}$ and $A_{i+}=\{x\in2^{[n]\backslash\{i\}}:x\cup\{i\}\in A\}$.
Clearly, $A=A_{i-}\cup (A_{i+}+\{i\})$ where $A+\{i\}=\{x\cup\{i\}:x\in A\}$.
Let us define $i$-compression  of $A$, denoted by $S_i(A)$, as follows:
$S_i(A)_{i-}=\mathcal{I}_{|A_{i-}|}\subseteq 2^{[n]\backslash\{i\}}$, $S_i(A)_{i+}=\mathcal{I}_{|A_{i+}|}\subseteq 2^{[n]\backslash\{i\}}$ and
$S_i(A)=S_i(A)_{i-}\cup (S_i(A)_{i+}+\{i\})=\mathcal{I}_{|A_{i-}|}\cup (\mathcal{I}_{|A_{i+}|}+\{i\})$. Thus either $S_i(A)=A$ or $S_i(A)<A$ in the simplicial ordering of sets.  We say that $A$ is $i$-compressed if and only if $A=S_i(A)$.

Notations and terminologies not mentioned here are as in \cite{west2005introduction}.
\section{Open Problem}
Let us recall the clique number and bounds for the b-chromatic number of powers of hypercubes given in \cite{francis2017b}.
\begin{figure}[t]
	\centering
	\includegraphics[angle=270,origin=c,
	width=1.00\textwidth]{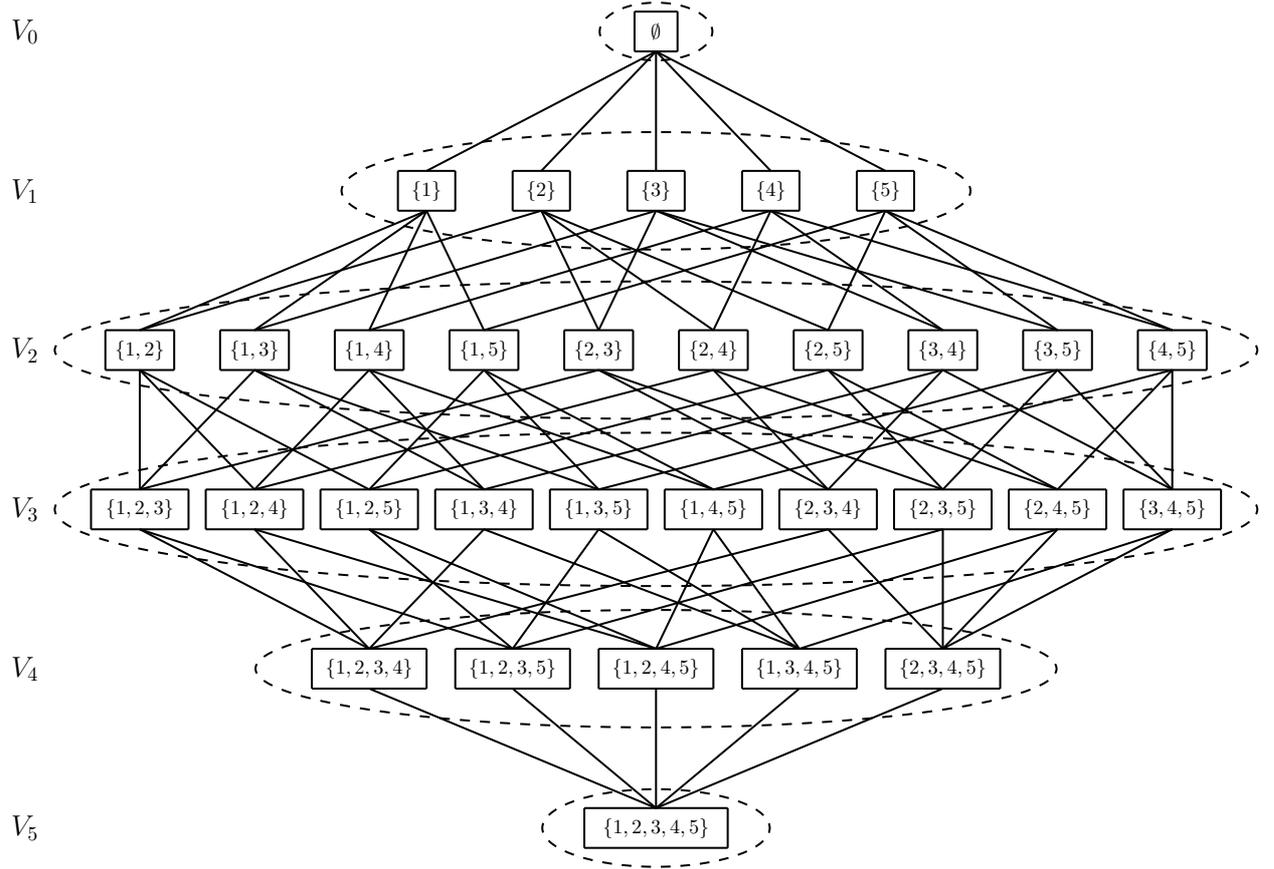}
	\vskip-1cm
	\caption{$Q_5$ where the boxes represent the vertices which are inside}
	\label{q5}
\end{figure}
\begin{thm}[\cite{francis2017b}]\label{hypercubeboundold}
(i) For $n\geq3$ and $1\leq p\leq n-1$, the clique size of $Q_n^p$ is

$\omega(Q_n^p)=\left\{\begin{array}{ll}
~\mathop{\sum}\limits_{i=0}^{\frac{p}{2}}\binom{n}{i}		& \mathrm{if}~p~ \mathrm{is~ even} \\
2 \mathop{\sum}\limits_{i=0}^{\frac{p-1}{2}}\binom{n-1}{i}	& \mathrm{if}~ p~ \mathrm{is~ odd}.
\end{array}
\right. $

\noindent (ii) For all $n\geq2$ and $\left\lfloor\frac{n}{2}\right\rfloor < p< n-1$, the b-chromatic number of $Q_n^p$ is

$2^{n-1}\leq b(Q_n^p)\leq2^{n-1}+\left\lfloor\frac{\omega(Q_n^p)}{2}\right\rfloor $.
\end{thm}
 In \cite{francis2017b}, the authors have also obtained the maximum number of common neighbors for a clique of size $2$ in $Q_n^p$.
Also, for a clique of larger size, they expected that the number of common neighbors will be maximum if the vertices of the clique in $Q_n^p$ are chosen as an initial segment in the simplicial order. They mentioned this as an open problem.
\begin{prob}[\cite{francis2017b}]\label{problem1}
Let $\mathcal{F}\subseteq 2^{[n]}$. 
Suppose $|\mathcal{F}|=m$, for some $2\leq m\leq2^{n-1}$, then for what kind of $\mathcal{F}$ will the $|C^p[\mathcal{F}]|$ be maximum?
\end{prob}
In Section \ref{problemsolved}, we have answered Problem \ref{problem1} and as expected, we have shown that $|C^p[\mathcal{F}]|$ is maximum when $\mathcal{F}$ is chosen as an initial segment in the simplicial ordering.
As a consequence, for $\left\lfloor\frac{n}{2}\right\rfloor < p< n-1$, we have obtained an upper bound, better than the existing bound given in Theorem \ref{hypercubeboundold}, for the b-chromatic number of the powers of hypercubes.
This can be seen in Section \ref{bcolorhypercube}.
In addition, we have investigated the b-chromatic number of some powers of Hamming graph, a generalization of the hypercubes.
\section{An affirmative answer to Problem \ref{problem1}}\label{problemsolved}
Let us start Section \ref{problemsolved} by showing that the set of all common neighbors of an initial segment of a power of hypercube is again an initial segment of the power of hypercube.
\begin{lem}\label{simplicialcommonngbr}
For any $n,p\in\mathbb{N}$, $a\in[2^n]$, there exists an integer $b\in[2^n]	$ such that $C^p[\mathcal{I}_a]=\mathcal{I}_b$.
\end{lem}
\begin{proof}
It is enough to prove that for any $x,y\in 2^{[n]}$, if $x\notin C^p[\mathcal{I}_a]$ and $x<y$, then $y\notin C^p[\mathcal{I}_a]$.
Since $x\notin C^p[\mathcal{I}_a]$, there exists a set $z\in \mathcal{I}_a$ such that $|x\triangle z|\geq p+1$.
Let us divide the proof into cases depending on $x\cap z$ and show that $y\notin C^p[\mathcal{I}_a]$.

\noindent \textbf{Case 1} $x\cap z=\emptyset$.

Clearly, $z\subseteq x^c$ and $|x\triangle z|=|x|+|z|\leq n$.
It is easy to observe that $y^c<x^c$.
Also, in simplicial ordering, for any integer $q<|y^c|$, the set containing the first $q$ elements of $y^c$ is equal or comes before than the set containing the first $q$ elements of $x^c$.
Moreover, the set containing the first $q$ elements of $y^c$ is equal or comes before than the set containing any $q$ elements of $x^c$.
Let $k=\max\{0,(|x|+|z|-|y|)\}$ and let us consider the set  containing the first $k$ elements in $y^c$ and call it $z'$.
Since $|y^c|=n-|y|\geq|x|+|z|-|y|$, $z'$ is a well defined set.
If $z'=\emptyset$, then $z'\in \mathcal{I}_a$. If $z'\neq\emptyset$, then $|z'|=|x|+|z|-|y|=|z|-(|y|-|x|)\leq|z|$.
Since $z$ is the set containing $|z|$ elements of $x^c$, we have either $z'=z$ or $z'<z$.
In both cases, $z'\in \mathcal{I}_a$.
Since $z'\subseteq y^c$, we have $z'\cap y=\emptyset$ and $|z'\triangle y|=|z'|+|y|\geq|x|+|z|-|y|+|y|=|x|+|z|\geq p+1$.
Thus $y\notin C^p[\mathcal{I}_a]$.

\noindent \textbf{Case 2} $x\cap z\neq\emptyset$.

Let us define $z_0=z\backslash (x\cap z)$. Clearly, $z_0<z$ and $|x\triangle z_0|>|x\triangle z|\geq p+1$. Thus $z_0\in \mathcal{I}_a$ such that $x\cap z_0=\emptyset$ and $|x\triangle z_0|\geq p+1$. By using Case 1, here also we can show that $y\notin C^p[\mathcal{I}_a]$.
\end{proof}
Let us recall a result due to Tsukerman \cite{tsukerman2013isoperimetric} which is useful to prove our next result.
\begin{thm}\label{allicompressed}\cite{tsukerman2013isoperimetric}
For $B\subseteq2^{[n]}$, if $B$ is $i$-compressed for each $i\in[n]$, but not an initial segment, then $|B|=2^{n-1}$ and $B$ is of the following form

%

$B=\left\{\begin{array}{ll}
 \mathcal{I}_\ell\backslash\{\{\frac{n+3}{2},\frac{n+5}{2},\ldots,n\}\} &$where $n$ is odd and $\ell=\sum\limits_{i=0}^{\frac{n-1}{2}}{\binom{n}{i}}+1
   \\
  \mathcal{I}_{\ell'}\backslash\{\{1,\frac{n}{2}+2,\frac{n}{2}+3,\ldots,n\}\} &$where $n$ is even and $\ell'=\sum\limits_{i=0}^{\frac{n}{2}-1}{\binom{n}{i}}+\binom{n-1}{n/2-1}+1.   \\
  \end{array}\right.$
%
%
\end{thm}
By using Lemma \ref{simplicialcommonngbr} and Theorem \ref{allicompressed}, let us obtain a positive answer to the Problem \ref{problem1}.
\begin{thm}\label{close}
For $n,p\in\mathbb{N}$, if $A\subseteq2^{[n]}$, 
then $|C^p[A]|\leq|C^p[\mathcal{I}_{|A|}]|$.
\end{thm}
\begin{proof}
Let $n,p\in\mathbb{N}$ and $A\subseteq2^{[n]}$. 
Let us prove the result by induction on `$n$'.
For $n=1,2$,  we can easily see that  $|C^p[A]|=|C^p[\mathcal{I}_{|A|}]|$.
Now, let us assume that the result is true for $n=q-1$, where $q\geq3$.
That is, for any $p\in\mathbb{N}$ and $A\subseteq 2^{[q-1]}$
, we have $|C^p[A]|\leq|C^p[\mathcal{I}_{|A|}]|$.
Let us prove the result for $n=q$.
If $p\geq q$ and $A\subseteq 2^{[q]}$, then $|C^p[A]|=2^{[q]}=|C^p[\mathcal{I}_{|A|}]|$.
So let us consider $1\leq p\leq q-1$.
Let $i\in[q]$. Clearly, $S_i(A)$, the $i$-compression of $A$ is such that $|A|=|S_i(A)|$.

\noindent \textbf{Claim:} $C^p[A]=\big(C^{p-1}[A_{i+}]\cap C^p[A_{i-}]\big)\cup\big((C^p[A_{i+}]\cap C^{p-1}[A_{i-}])+\{i\}\big)$

Note that, if $B$ is a set considered in $2^{[\ell]}$, then $C^p[B]$ will also be considered with respect to $2^{[\ell]}$. 
Therefore in the above claim, $C^p[A]$ will be a set in $2^{[q]}$ and the other four $C^p$ on the right side will be sets in $2^{[q]\backslash\{i\}}\cong 2^{[q-1]}$.

\noindent Let $x\in C^p[A]$. Let us consider the cases $i\notin x$ and $i\in x$ separately.

%
\noindent \textbf{Case 1} $i\notin x$.

Here, $x\in C^p[A]$

$\Longleftrightarrow$ for every $y\in A=(A_{i+}+\{i\})\cup A_{i-}$, $|x\triangle y|\leq p$.

$\Longleftrightarrow$ for every $w\in A_{i+}$, $|x\triangle (w\cup\{i\})|\leq p$ and for every $z\in A_{i-}$, $|x\triangle z|\leq p$.

$\Longleftrightarrow$ for every $w\in A_{i+}$, $|x\triangle w|\leq p-1$ and for every $z\in A_{i-}$, $|x\triangle z|\leq p.$

$\Longleftrightarrow x\in C^{p-1}[A_{i+}] $ and $ x\in C^p[A_{i-}]$.

$\Longleftrightarrow x\in C^{p-1}[A_{i+}] \cap C^p[A_{i-}]$.

\noindent \textbf{Case 2} $i\in x$.

In this case, $x\in C^p[A]$

$\Longleftrightarrow$ for every $y\in A$, $|x\triangle y|\leq p$.

$\Longleftrightarrow$ for every $w\in A_{i+}$, $|x\triangle (w\cup\{i\})|\leq p$ and for every $z\in A_{i-}$, $|x\triangle z|\leq p$.

$\Longleftrightarrow$ for every $w\in A_{i+}$, $|(x\backslash\{i\})\triangle w|\leq p$ and for every $z\in A_{i-}$, $|(x\backslash\{i\})\triangle z|\leq p-1$.

$\Longleftrightarrow x\backslash\{i\}\in C^{p}[A_{i+}] $ and $x\backslash\{i\}\in C^{p-1}[A_{i-}]$.

$\Longleftrightarrow x\backslash\{i\}\in C^p[A_{i+}]\cap C^{p-1}[A_{i-}]$.

$\Longleftrightarrow x\in (C^p[A_{i+}]\cap C^{p-1}[A_{i-}])+\{i\}$.

Hence for any $A\subseteq 2^{[n]}$, $C^p[A]=\big(C^{p-1}[A_{i+}]\cap C^p[A_{i-}]\big)\cup\big((C^p[A_{i+}]\cap C^{p-1}[A_{i-}])+\{i\}\big)$. 
By replacing $A$ with $S_i(A)$ we observe  that

\noindent$\begin{array}{ll}
C^p[S_i(A)] &=\big(C^p[S_i(A)_{i-}]\cap C^{p-1}[S_i(A)_{i+}]\big)\cup\big((C^p[S_i(A)_{i+}]\cap C^{p-1}[S_i(A)_{i-}])+\{i\}\big)\\
&=\big(C^p[\mathcal{I}_{|A_{i-}|}]\cap C^{p-1}[\mathcal{I}_{|A_{i+}|}]\big)\cup\big((C^p[\mathcal{I}_{|A_{i+}|}]\cap C^{p-1}[\mathcal{I}_{|A_{i-}|}])+\{i\}\big)\\
\end{array}$

Since $A_{i-}, A_{i+}\subseteq 2^{[q]\backslash i}\cong 2^{[q-1]}$, by our assumption we see that $|C^p[A_{i-}]|\leq|C^p[\mathcal{I}_{|A_{i-}|}]|$, $|C^p[A_{i+}]|\leq|C^p[\mathcal{I}_{|A_{i+}|}]|$, $|C^{p-1}[A_{i-}]|\leq|C^{p-1}[\mathcal{I}_{|A_{i-}|}]|$ and $|C^{p-1}[A_{i+}]|\leq|C^{p-1}[\mathcal{I}_{|A_{i+}|}]|$.
By using Lemma \ref{simplicialcommonngbr}, both $C^p[\mathcal{I}_{|A_{i-}|}]$ and $C^{p-1}[\mathcal{I}_{|A_{i+}|}]$ are initial segments of $2^{[q]\backslash\{i\}}$ and hence one of these sets is a subset of the other.
Similarly, for $C^p[\mathcal{I}_{|A_{i+}|}]$ and $C^{p-1}[\mathcal{I}_{|A_{i-}|}]$, one of these is a subset of the other.
Thus we see that 

\noindent $|C^p[A_{i-}]\cap C^{p-1}[A_{i+}]|\leq\min{\{|C^p[\mathcal{I}_{|A_{i-}|}]|,|C^{p-1}[\mathcal{I}_{|A_{i+}|}]|\}}=|C^p[\mathcal{I}_{|A_{i-}|}]\cap C^{p-1}[\mathcal{I}_{|A_{i+}|}]|$ and \\
$|C^p[A_{i+}]\cap C^{p-1}[A_{i-}]|\leq\min{\{|C^p[\mathcal{I}_{|A_{i+}|}]|,|C^{p-1}[\mathcal{I}_{|A_{i-}|}]|\}}=|C^p[\mathcal{I}_{|A_{i+}|}]\cap C^{p-1}[\mathcal{I}_{|A_{i-}|}]|$.

\begin{equation}
\begin{aligned}[b]
\textnormal{Hence, }|C^p[A]| &=|C^p[A_{i-}]\cap C^{p-1}[A_{i+}]|+|C^p[A_{i+}]\cap C^{p-1}[A_{i-}]|\\
&\leq |C^p[\mathcal{I}_{|A_{i-}|}]\cap C^{p-1}[\mathcal{I}_{|A_{i+}|}]|+|C^p[\mathcal{I}_{|A_{i+}|}]\cap C^{p-1}[\mathcal{I}_{|A_{i-}|}]|\\
&=|C^p[S_i(A)]|
\end{aligned}\label{asia}
\end{equation}

Let $A_0=A$ and $A_1=S_i(A)$. 
If there exists an integer $j\in[q]$ such that $A_1$ is not $j$-compressed,
 define $A_2=S_j(A_1)$. 
Again, if there exists an integer $\ell\in[q]$ such that $A_2$ is not $\ell$-compressed,
 define $A_3=S_\ell(A_2)$. 
 Continue this process until there exists an integer $k\in\mathbb{N}$ such that $A_k$ is $t$-compressed for every $t\in[q]$.
Since $A=A_0>A_1>A_2>\cdots$ in simplicial ordering of sets and $\left|2^{\left[2^{[q]}\right]}\right|$ is finite, we will get an integer $k$ such that $A_k$ is $t$-compressed for every $t\in[q]$. 
Also for any $j\in\{1,2,\ldots,k\}$, since $A_j$ is some $t$-compression of $A_{j-1}$, equation \ref{asia} will imply that $|C^p[A_{j-1}]|\leq |C^p[A_j]|$. 
Thus $|C^p[A_{0}]|\leq |C^p[A_1]|\leq |C^p[A_2]|\leq \cdots \leq |C^p[A_k]|$. 
In addition, $|A|=|A_{0}|=|A_1|=\cdots=|A_k|$. 
Finally, if  $|C^p[A_k]|\leq |C^p[\mathcal{I}_{|A_k|}]|= |C^p[\mathcal{I}_{|A|}]|$, we are done. 
So, let us  establish that $|C^p[A_k]|\leq |C^p[\mathcal{I}_{|A_k|}]|$.   
Since $A_k$ is $t$-compressed for every $t\in[q]$, by using Theorem \ref{allicompressed}, either $A_k=\mathcal{I}_{|A|}$ or $A_k$ is of the form given in Theorem \ref{allicompressed}.
If $A_k=\mathcal{I}_{|A|}$, there is nothing to prove.
Next, let us consider $A_k$ as given in Theorem \ref{allicompressed}.

For $q$ being odd, $A_k=\mathcal{I}_\ell\backslash\big{\{}\{\frac{q+3}{2},\frac{q+5}{2},\ldots,q\}\big{\}}$, where $\ell=\sum\limits_{i=0}^{\frac{q-1}{2}}{\binom{q}{i}}+1$.
Let us begin by considering $C^p[\mathcal{I}_{\ell-1}]$ and move towards $C^p[A_k]$.
Without much difficulty, one can observe that $C^p[\mathcal{I}_{\ell-1}]=\bigcup\limits_{i=0}^{p-\frac{q-1}{2}}{V_i}$. 
Let us find the smallest collection of sets in $V_{\frac{q-1}{2}}$, say $U$, such that \linebreak $C^p[\mathcal{I}_{\ell-1}\backslash U]=C^p[\mathcal{I}_{\ell-1}]\cup \big{\{}\{1,2,\ldots,p-\frac{q-1}{2}+1\}\big{\}}$, that is, we would like to know how many sets are to be removed from $\mathcal{I}_{\ell-1}$ if one wants to increase $C^p[\mathcal{I}_{\ell-1}]$ by one set in simplicial ordering, namely $\{1,2,\ldots,p-\frac{q-1}{2}+1\}$. 
Clearly, $U$ will be the set containing  the sets  $x\in2^{[q]}$ such that $x\cap\{1,2,\ldots,p-\frac{q-1}{2}+1\}=\emptyset$.
Hence $|U|=\binom{q-(p-\frac{q-1}{2}+1)}{\frac{q-1}{2}}$.
When $p=q-1$, $|U|=1$ and $C^p[\mathcal{I}_{\ell-1}]=\mathcal{I}_{\ell-1}$. 
Therefore $C^p\left[\mathcal{I}_{\ell-1}\backslash\big{\{}\{\frac{q+3}{2},\frac{q+5}{2},\ldots,q\}\big{\}}\right]=C^p[\mathcal{I}_{\ell-1}]\cup\big{\{}\{1,2,\ldots,\frac{q+1}{2}\}\big{\}}=\mathcal{I}_{\ell-1}\cup\big{\{}\{1,2,\ldots,\frac{q+1}{2}\}\big{\}}$.
Also, one can see that $C^p[A_k]=
\mathcal{I}_{\ell-1}\cup\big{\{}\{1,2,\ldots,\frac{q+1}{2}\}\big{\}}
\backslash\big{\{}\{\frac{q+3}{2},\frac{q+5}{2},\ldots,q\}\big{\}}={A}_k$. 
Thus $|C^p[A_k]|=|A_k|=\ell-1=|\mathcal{I}_{\ell-1}|=|C^p[\mathcal{I}_{\ell-1}]|$.
When $p<q-1$, $|U|\geq\frac{q+1}{2}$.
Therefore $C^p\left[\mathcal{I}_{\ell-1}\right]=C^p\left[\mathcal{I}_{\ell-1}\backslash\big\{\{\frac{q+3}{2},\frac{q+5}{2},\ldots,q\}\big\}\right]\supseteq C^p[A_k]$.
Thus $|C^p[A_k]|\leq|C^p[\mathcal{I}_{\ell-1}]|$.

For $q$ being even, $A_k=\mathcal{I}_{\ell'}\backslash\{\{1,\frac{q}{2}+2,\frac{q}{2}+3,\ldots,q\}\}$, where $\ell'=\sum\limits_{i=0}^{\frac{q}{2}-1}{\binom{q}{i}}+\binom{q-1}{\frac{q}{2}-1}+1$.
Here also, let us begin with $C^p[\mathcal{I}_{\ell'-1}]$ and move towards $C^p[A_k]$.
Clearly, one can see that $C^p[\mathcal{I}_{\ell'-1}]=\left(\bigcup\limits_{i=0}^{p-\frac{q}{2}}{V_i}\right)\cup\{x\in V_{(p-\frac{q}{2}+1)}:1\in x\}$. 
As done in the odd case, let us find the smallest collection of sets in $V_{\frac{q}{2}}$, say $U'$, such that $C^p[\mathcal{I}_{\ell'-1}\backslash U']=C^p[\mathcal{I}_{\ell'-1}]\cup\big\{\{2,3,\ldots,p-\frac{q}{2}+2\}\big\}$.
Clearly, $U'$ will be the set containing  all  $x\in2^{[q]}$ such that $1\in x$ and $x\cap\{2,3,\ldots,p-\frac{q}{2}+2\}=\emptyset$.
Hence $|U'|=\binom{q-(p-\frac{q}{2}+2)}{\frac{q}{2}-1}$.
When $p=q-1$, $|U'|=1$ and $C^p[\mathcal{I}_{\ell'-1}]=\mathcal{I}_{\ell'-1}$. 
Therefore $C^p\left[\mathcal{I}_{\ell'-1}\backslash\big{\{}\{1,\frac{q}{2}+2,\frac{q}{2}+3,\ldots,q\}\big{\}}\right]=C^p[\mathcal{I}_{\ell'-1}]\cup\big{\{}\{2,3,\ldots,p-\frac{q}{2}+2\}\big{\}}=\mathcal{I}_{\ell'-1}\cup\big{\{}\{2,3,\ldots,p-\frac{q}{2}+2\}\big{\}}$.
Also one can see that $C^p[A_k]=
\mathcal{I}_{\ell'-1}\cup\big{\{}\{2,3,\ldots,\frac{q}{2}+1\}\big{\}}
\backslash\big{\{}\{1,\frac{q}{2}+2,\frac{q}{2}+3,\ldots,q\}\big{\}}={A}_k$. 
Thus $|C^p[A_k]|=|A_k|=\ell'-1=|\mathcal{I}_{\ell'-1}|=|C^p[\mathcal{I}_{\ell'-1}]|$.
When $p<q-1$, $|U'|\geq\frac{q}{2}$.
Therefore $C^p\left[\mathcal{I}_{\ell'-1}\right]=C^p\left[\mathcal{I}_{\ell'-1}\backslash\big\{\{1,\frac{q}{2}+2,\frac{q}{2}+3,\ldots,q\}\big\}\right]\supseteq C^p[A_k]$.
Thus $|C^p[A_k]|\leq|C^p[\mathcal{I}_{\ell'-1}]|$.
In both cases, we have showed that $|C^p[A_k]|\leq|C^p[\mathcal{I}_{|A|}]|$.
\end{proof}
Note that, Theorem \ref{close} cannot be said for $C^p(A)$ and $C^p(\mathcal{I}_{|A|})$. 
This is because $\mathcal{I}_{|A|}\subseteq C^p[\mathcal{I}_{|A|}]$ and with no condition on $A$, $A$ may not be a subset of $C^p[A]$. 
Thus even though $|C^p[A]|\leq|C^p[\mathcal{I}_{|A|}]|$, we cannot say that $|C^p(A)|\leq|C^p(\mathcal{I}_{|A|})|$. 
To achieve this, we shall impose the condition that $A\subseteq2^{[n]}$ such that for every pair $x,y\in A, |x\triangle y|\leq p$. 
This can be seen in Corollary \ref{open}.
\begin{cor}\label{open}
For $n,p\in\mathbb{N}$, if $A\subseteq2^{[n]}$ such that for every pair $x,y\in A, |x\triangle y|\leq p$, 
then $|C^p(A)|\leq\left|C^p\left(\mathcal{I}_{|A|}\right)\right|$.
\end{cor}
\begin{proof}
By using Theorem \ref{close}, $|C^p[A]|\leq\left|C^p\left[\mathcal{I}_{|A|}\right]\right|$.
Since $A\subseteq2^{[n]}$ such that for every pair $x,y\in A$, $|x\triangle y|\leq p$, we see that $A\subseteq C^p[A]$ and hence $\left|C^p\left(A\right)\right|=\left|C^p\left[A\right]\right|-|A|$. 
Also $\left|C^p\left(\mathcal{I}_{|A|}\right)\right|\geq\left|C^p\left[\mathcal{I}_{|A|}\right]\right|-|\mathcal{I}_{|A|}|=\left|C^p\left[\mathcal{I}_{|A|}\right]\right|-|A|$.
Hence, $|C^p(A)|=\left|C^p\left[A\right]\right|-|A|\leq\left|C^p\left[\mathcal{I}_{|A|}\right]\right|-|A|\leq\left|C^p\left(\mathcal{I}_{|A|}\right)\right|$.
\end{proof}
Next, we find the number of common neighbors for some very particular initial segments of some powers of hypercubes. These will be useful in establishing a better upper bound for the b-chromatic number of some powers of hypercubes. This can be seen in Theorem \ref{newboundhypercube}.
\begin{thm}\label{x-k}
For $n$ being odd and $\frac{n+1}{2}\leq p\leq n-2$, if $r=\sum\limits_{i=0}^{p-\frac{n-1}{2}}{\binom{n}{i}}$ and $s=\binom{p}{p-\frac{n-1}{2}}$, then $|C^p(\mathcal{I}_{(r-s)})|=2^{n-1}-(r-2s)$.
For $n$ being even and $\frac{n}{2}+1\leq p\leq n-2$, if $r'=\sum\limits_{i=0}^{p-\frac{n}{2}}{\binom{n}{i}}+ \binom{n-1}{p-\frac{n}{2}}$ and $s'=\binom{p-1}{p-\frac{n}{2}}$, then $|C^p(\mathcal{I}_{(r'-s')})|=2^{n-1}-(r'-2s')$.
\end{thm}
\begin{proof}
For $n$ being odd and $\frac{n+1}{2}\leq p\leq n-2$, let us first prove that $C^p[\mathcal{I}_r]=\mathcal{I}_{2^{n-1}}$.
Let $b\in \mathcal{I}_{2^{n-1}}$. 
Since $2^{n-1}=\sum\limits_{i=0}^{\frac{n-1}{2}}{\binom{n}{i}}$, we see that $|a\triangle b|\leq (p-\frac{n-1}{2})+\frac{n-1}{2}\leq p$, for every $a\in \mathcal{I}_r$.
Thus $\mathcal{I}_{2^{n-1}}\subseteq C^p[\mathcal{I}_r]$.
Since $\{\frac{n+3}{2},\frac{n+5}{2},\ldots,p+1\}\in\mathcal{I}_r$ and $\{1,2,\ldots,\frac{n+1}{2}\}$, the first element in the simplicial ordering next to $\mathcal{I}_{2^{n-1}}$ have symmetric difference $p+1$, we see that $\{1,2,\ldots,\frac{n+1}{2}\}\notin C^p[\mathcal{I}_r]$.
By using Lemma \ref{simplicialcommonngbr}, we see that $C^p[\mathcal{I}_r]\subseteq\mathcal{I}_{2^{n-1}}$.
Hence $C^p[\mathcal{I}_r]=\mathcal{I}_{2^{n-1}}$.
Note that for $t\leq r$, $C^p[\mathcal{I}_t]\supseteq C^p[\mathcal{I}_r]$.
Since $s=\binom{p}{p-\frac{n-1}{2}}<\binom{n}{p-\frac{n-1}{2}}$, we have $C^p(\mathcal{I}_{(r-s)})=(\mathcal{I}_{2^{n-1}}\setminus\mathcal{I}_{r-s})\cup (C^p[\mathcal{I}_{(r-s)}]\cap V_{\frac{n+1}{2}})$.
Thus $|C^p(\mathcal{I}_{(r-s)})|=2^{n-1}-(r-s)+|C^p[\mathcal{I}_{(r-s)}]\cap V_{\frac{n+1}{2}}|$.
So let us find the value of $|C^p[\mathcal{I}_{(r-s)}]\cap V_{\frac{n+1}{2}}|$.
Clearly $\mathcal{I}_{(r-s)}=\left(\bigcup\limits_{0\leq i<p-\frac{n-1}{2}} V_{i}\right)\cup\big\{a\in V_{p-\frac{n-1}{2}}: a$ has at least one of the elements from $\{1,2,\ldots,n-p\}\big\}$.
Let $b\in V_{\frac{n+1}{2}}$ such that $\{1,2,\ldots,n-p\}\subseteq b$, then for any element $a\in \mathcal{I}_{(r-s)}\cap V_{(p-\frac{n-1}{2})}$,\ $a\cap b\neq\emptyset$ and $|a\triangle b|\leq (p-\frac{n-1}{2}-1)+(\frac{n+1}{2}-1)=p-1$. Hence $b\in C^p[\mathcal{I}_{(r-s)}]$ and $|C^p[\mathcal{I}_{(r-s)}]\cap V_{\frac{n+1}{2}}|\geq \binom{n-(n-p)}{\frac{n+1}{2}-(n-p)}=\binom{p}{p-\frac{n-1}{2}}=s$.
Let us consider the $(s+1)^{\textnormal{th}}$ element in $V_{\frac{n+1}{2}}$, namely $\left\{1,2,\ldots,n-p-1,n-p+1,n-p+2,\ldots,\frac{n+3}{2}\right\}$ and  $\left\{n-p,\frac{n+5}{2},\frac{n+7}{2},\ldots,p+1\right\}\in \mathcal{I}_{(r-s)}\cap V_{p-\frac{n-1}{2}}$. 
Since the symmetric difference between these two sets is $p+1$, $\{1,2,\ldots,n-p-1,n-p+1,n-p+2,\ldots,\frac{n+3}{2}\}\notin C^p[\mathcal{I}_{(r-s)}]$ and hence  by using Lemma \ref{simplicialcommonngbr}, we see that $|C^p[\mathcal{I}_{(r-s)}]\cap V_{\frac{n+1}{2}}|=s$.
Therefore, $|C^p(\mathcal{I}_{(r-s)})|=2^{n-1}-(r-s)+s=2^{n-1}-(r-2s)$.

Next, for $n$ being even and $\frac{n}{2}+1\leq p\leq n-2$, the proof will be similar to the one given for odd case. Let us first prove that  $C^p[\mathcal{I}_{r'}]=\mathcal{I}_{2^{n-1}}$.
Let $b\in \mathcal{I}_{2^{n-1}}$.
Since $n$ is even, $2^{n-1}=\sum\limits_{i=0}^{\frac{n}{2}-1}{\binom{n}{i}}+ \binom{n-1}{\frac{n}{2}-1}$ and hence $\mathcal{I}_{2^{n-1}}=\left(\bigcup\limits_{0\leq i<\frac{n}{2}} V_{i}\right)\cup\{a\in V_{\frac{n}{2}}: a$ contains $1\}$. 
Also $\mathcal{I}_{r'}=\left(\bigcup\limits_{0\leq i\leq p-\frac{n}{2}} V_{i}\right)\cup \{a\in V_{p-\frac{n}{2}+1}: a$ contains $1\}$. 
Thus for  $a\in \mathcal{I}_{r'}$,  $|a\triangle b|\geq p+1$ can happen only when $b\in V_{\frac{n}{2}}$ and $a\in V_{p-\frac{n}{2}+1}$. 
But in this case, $1\in a\cap b$ and therefore $|a\triangle b|\leq (p-\frac{n}{2})+\frac{n}{2}-1\leq p-1$. 
Hence  $\mathcal{I}_{2^{n-1}}\subseteq C^p[\mathcal{I}_{r'}]$. 
%
%
Also, the first element in the simplicial ordering that comes next to $\mathcal{I}_{2^{n-1}}$, namely $\{2,3,\ldots,\frac{n}{2}+1\}$ and $\{1,\frac{n}{2}+2,\frac{n}{2}+3,\ldots,p+1\}\in \mathcal{I}_{r'}$ have symmetric difference $p+1$ and hence by using Lemma \ref{simplicialcommonngbr}, we see that $C^p[\mathcal{I}_{r'}]=\mathcal{I}_{2^{n-1}}$.
Since $s'=\binom{p-1}{p-\frac{n}{2}}<\binom{n-1}{p-\frac{n}{2}}$, we have $|C^p(\mathcal{I}_{({r'}-s')})|=2^{n-1}-({r'}-s')+|C^p[\mathcal{I}_{({r'}-s')}]\cap V_{\frac{n}{2}}\cap \overline{\mathcal{I}_{2^{n-1}}}|$, where $\overline{A}$ denotes the complement of $A$.
So let us find $|C^p[\mathcal{I}_{({r'}-s')}]\cap V_{\frac{n}{2}}\cap\overline{\mathcal{I}_{2^{n-1}}}|$.
Clearly $\mathcal{I}_{(r'-s')}=\left(\bigcup\limits_{0\leq i\leq p-\frac{n}{2}} V_{i}\right)\cup\left\{a\in V_{p-\frac{n}{2}+1}: 1\in a\ \textnormal{and $a$ has at least one of the elements from }\{2,3,\ldots,n-p+1\}\right\}$.
Let $b\in V_{\frac{n}{2}}\cap\overline{\mathcal{I}_{2^{n-1}}}$ such that $\{2,3,\ldots,n-p+1\}\subseteq b$, then for any $a\in \mathcal{I}_{(r'-s')}\cap V_{(p-\frac{n}{2}+1)}$, $a\cap b\neq\emptyset$ and hence $|a\triangle b|\leq (p-\frac{n}{2})+(\frac{n}{2}-1)=p-1$.
Therefore $|C^p[\mathcal{I}_{(r'-s')}]\cap V_{\frac{n}{2}}\cap\overline{\mathcal{I}_{2^{n-1}}}|\geq \binom{(n-1)-(n-p)}{\frac{n}{2}-(n-p)}=\binom{p-1}{p-\frac{n}{2}}=s'$.
Let us consider the $(s'+1)^{\textnormal{th}}$ element in $V_{\frac{n}{2}}\cap\overline{\mathcal{I}_{2^{n-1}}}$, namely $\{2,3,\ldots,n-p,$ $n-p+2,n-p+3,\ldots,\frac{n}{2}+2\}$ and  $\left\{1,n-p+1,\frac{n}{2}+3,\frac{n}{2}+4,\ldots,p+1\right\}\in \mathcal{I}_{(r'-s')}\cap V_{p-\frac{n}{2}+1}$. 
Since the symmetric difference between these two sets is $p+1$, $\{2,3,\ldots,n-p,n-p+2,n-p+3,\ldots,\frac{n}{2}+2\}\notin C^p[\mathcal{I}_{(r'-s')}]$ and hence  by using Lemma \ref{simplicialcommonngbr}, we see that $|C^p[\mathcal{I}_{(r'-s')}]\cap V_{\frac{n}{2}}\cap\overline{\mathcal{I}_{2^{n-1}}}|=s'$.
Therefore, $|C^p(\mathcal{I}_{(r'-s')})|=2^{n-1}-(r'-s')+s'=2^{n-1}-(r'-2s')$.
\end{proof}
\begin{cor}\label{i2l}
For $n$ being odd and $\frac{n+1}{2}\leq p\leq n-2$, if $r=\sum\limits_{i=0}^{p-\frac{n-1}{2}}{\binom{n}{i}}$ and $s=\binom{p}{p-\frac{n-1}{2}}$, then $|C^p(\mathcal{I}_{(r-s+1)})|<2^{n-1}-(r-2s+1)$.
For $n$ being even and $\frac{n}{2}+1\leq p\leq n-2$, if $r'=\sum\limits_{i=0}^{p-\frac{n}{2}}{\binom{n}{i}}+ \binom{n-1}{p-\frac{n}{2}}$ and $s'=\binom{p-1}{p-\frac{n}{2}}$, then $|C^p(\mathcal{I}_{(r'-s'+1)})|<2^{n-1}-(r'-2s'+1)$.
\end{cor}
\begin{proof}
For $n$ being odd, by arguments similar to that given in Theorem \ref{x-k}, we can see that $|C^p(\mathcal{I}_{(r-s+1)})|=2^{n-1}-(r-s+1)+|C^p[\mathcal{I}_{(r-s+1)}]\cap\overline{\mathcal{I}_{2^{n-1}}}|$ and $|C^p[\mathcal{I}_{(r-s)}]\cap\overline{\mathcal{I}_{2^{n-1}}}|=s$.
One can observe that $\mathcal{I}_{(r-s+1)}\backslash \mathcal{I}_{(r-s)}=\{\{n-p+1,n-p+2,\ldots,\frac{n+1}{2}\}\}$.
So $C^p[\mathcal{I}_{(r-s+1)}]\cap\overline{\mathcal{I}_{2^{n-1}}}$ will not contain $\{1,2,\ldots,n-p,\frac{n+3}{2},\frac{n+5}{2},\ldots,p+1\}$ but contained in $C^p[\mathcal{I}_{(r-s)}]\cap\overline{\mathcal{I}_{2^{n-1}}}$.
Therefore, $|C^p[\mathcal{I}_{(r-s+1)}]\cap\overline{\mathcal{I}_{2^{n-1}}}|<s$ and $|C^p(\mathcal{I}_{(r-s+1)})|<2^{n-1}-(r-s+1)+s=2^{n-1}-(r-2s+1)$.
Similarly, for $n$ being even, one can prove $|C^p(\mathcal{I}_{(r'-s'+1)})|<2^{n-1}-(r'-2s'+1)$.
\end{proof}
\section{Bounds for the b-chromatic number of some powers of hypercubes}\label{bcolorhypercube}
Let us start this section by discussing a relationship between the b-chromatic number and the number of common neighbors of a clique in powers of hypercubes.
\begin{lem}\label{limp2l}
For $n,p,\ell \in\mathbb{N}$ with $\ell\leq2^{n-1}$, if $b(Q_n^p)\geq2^{n-1}+\ell$, then there exists a clique of size $2\ell$, say $A_{2\ell}$, such that $|C^p(A_{2\ell})|\geq2^{n-1}-\ell$.
\end{lem}
\begin{proof}
If $b(Q_n^p)\geq2^{n-1}+\ell$, then there exist at least $2\ell$ color classes which are singletons.
Suppose not, if there exist at most $2\ell-1$ color classes which are singletons, then there exist at least $2^{n-1}-\ell+1$ color classes which are not singletons.
 This in turn will imply that  $2^n\geq2\ell-1+2(2^{n-1}-\ell+1)=2\ell-1+2^{n}-2\ell+2=2^n+1$, a contradiction.
Now, let $A_{2\ell}$ denote a subset of $2\ell$ vertices from the singleton color classes.
Since in a b-coloring, every color class must contain a color dominating vertex and the color dominating vertices of the color classes other than the singleton color classes corresponding to $A_{2\ell}$ will be  common neighbors  to the vertices in $A_{2\ell}$. 
Therefore, $|C^p(A_{2\ell})|\geq2^{n-1}+\ell-2\ell=2^{n-1}-\ell$.
\end{proof}

As a simple consequence of Corollary \ref{open} and Lemma \ref{limp2l}, we can obtain an upper bound for the b-chromatic number of some powers of hypercubes which is better than the upper bound given in Theorem \ref{hypercubeboundold}.
\begin{thm}\label{roughbound}
	For any odd $n\geq5$ and $\frac{n+1}{2}\leq p\leq n-2$, if $r=\sum\limits_{i=0}^{p-\frac{n-1}{2}}{\binom{n}{i}}$
	, then $b(Q_n^p)\leq2^{n-1}+\left\lceil\frac{r}{2}\right\rceil-1$.
	For any even $n\geq6$ and $\frac{n}{2}+1\leq p\leq n-2$, if $r'=\sum\limits_{i=0}^{p-\frac{n}{2}}{\binom{n}{i}}+ \binom{n-1}{p-\frac{n}{2}}$
	, then $b(Q_n^p)\leq2^{n-1}+\left\lceil\frac{r'}{2}\right\rceil-1$.
\end{thm}
\begin{proof}
	Let us start with $n$ being odd and $n\geq 5$.
	Suppose $b(Q_n^p)\geq2^{n-1}+\left\lceil\frac{r}{2}\right\rceil$, then by using Lemma \ref{limp2l}, there exists a clique of size $2\left\lceil\frac{r}{2}\right\rceil$, say $A_{2\left\lceil\frac{r}{2}\right\rceil}$ such that $\left|C^p\left(A_{2\left\lceil\frac{r}{2}\right\rceil}\right)\right|\geq2^{n-1}-\left\lceil\frac{r}{2}\right\rceil$.
	By using Corollary \ref{open}, $\left|C^p\left(I_{2\left\lceil\frac{r}{2}\right\rceil}\right)\right|\geq\left|C^p\left(A_{2\left\lceil\frac{r}{2}\right\rceil}\right)\right|\geq 2^{n-1}-\left\lceil\frac{r}{2}\right\rceil$.
	Since $2\left\lceil\frac{r}{2}\right\rceil\geq r$, we can see that $|C^p(\mathcal{I}_{2\left\lceil\frac{r}{2}\right\rceil})|\leq|C^p(\mathcal{I}_r)|$.
	As discussed in Theorem \ref{x-k}, $C^p[\mathcal{I}_r]=\mathcal{I}_{2^{n-1}}$ and hence $|C^p(\mathcal{I}_r)|=2^{n-1}-r$.
	Therefore $2^{n-1}-\left\lceil\frac{r}{2}\right\rceil\leq\left|C^p\left(I_{2\left\lceil\frac{r}{2}\right\rceil}\right)\right|\leq|C^p(\mathcal{I}_r)|=2^{n-1}-r$, a contradiction.
	For $n$ being even, here also  $C^p[\mathcal{I}_{r'}]=\mathcal{I}_{2^{n-1}}$.
	So by using similar technique, one can prove that $b(Q_n^p)\leq2^{n-1}+\left\lceil\frac{r'}{2}\right\rceil-1$.
\end{proof}
Natural question one would get is ``Whether the upper bound given in Theorem \ref{roughbound} is optimal?" The answer is no. We are going to establish an upper bound which is even better than the one given in Theorem \ref{roughbound}. This is going to be slightly complicated and for doing this, we will need to discuss a few special cases separately. This is given in Lemmas \ref{parteg}-\ref{r>3s}.
\begin{lem}\label{parteg}
$b(Q_7^5)\leq2^6+9$
\end{lem}
\begin{proof}
Suppose $b(Q_7^5)\geq2^6+10$, by Lemma \ref{limp2l}, there exists a clique of size 20, say, $A_{20}$ such that $|C^5(A_{20})|\geq2^6-10$.
By Corollary \ref{open}, $|C^5(\mathcal{I}_{20})|\geq|C^5(A_{20})|\geq2^6-10$.
Let us find $C^5[\mathcal{I}_{20}]$ in $Q_7^5$.
$\mathcal{I}_{20}=\{\emptyset,\{1\},\{2\},\ldots,\{7\},\{1,2\},\{1,3\},\ldots,\{1,7\},\{2,3\},\{2,4\},\ldots,\{2,7\},\{3,4\}\}$.
So $C^5[\mathcal{I}_{20}]$ will contain all the vertices in $V_0,V_1,V_2,V_3$ and $\{1,2,3,4\},\{1,2,3,5\},\{1,2,3,6\},\linebreak\{1,2,3,7\},\{1,2,4,5\},\{1,2,4,6\},\{1,2,4,7\}$ in $V_4$.
Therefore, $|C^5(\mathcal{I}_{20})|=2^6+7-20=2^6-13<2^6-10$, a contradiction.
%
%
%
\end{proof}
\begin{lem}\label{q=1}
For odd $n\geq5$, $b(Q_n^{\frac{n+1}{2}})\leq2^{n-1}+\left\lfloor\frac{n+1}{4}\right\rfloor$.
\end{lem}
\begin{proof}
Suppose  $b(Q_n^{\frac{n+1}{2}})\geq2^{n-1}+\left\lfloor\frac{n+1}{4}\right\rfloor+1$, by using Lemma \ref{limp2l}, there exists a clique of size $2\left(\left\lfloor\frac{n+1}{4}\right\rfloor+1\right)$, say $A_{2\left(\left\lfloor\frac{n+1}{4}\right\rfloor+1\right)}$ such that $\left|C^{\frac{n+1}{2}}\left(A_{2\left(\left\lfloor\frac{n+1}{4}\right\rfloor+1\right)}\right)\right|\geq2^{n-1}-\left(\left\lfloor\frac{n+1}{4}\right\rfloor+1\right)$.
Since $2\left(\left\lfloor\frac{n+1}{4}\right\rfloor+1\right)\geq \frac{n+1}{2}+1$, $\left|C^{\frac{n+1}{2}}\left(\mathcal{I}_{{\frac{n+1}{2}}+1}\right)\right|\geq\left|C^{\frac{n+1}{2}}\left(\mathcal{I}_{2\left(\left\lfloor\frac{n+1}{4}\right\rfloor+1\right)}\right)\right|$.
One can easily observe that $\left|C^{\frac{n+1}{2}}\left(\mathcal{I}_{{\frac{n+1}{2}}+1}\right)\right|=2^{n-1}-\frac{n+1}{2}-1+1=2^{n-1}-\frac{n+1}{2}$.
By using Corollary \ref{open}, $\left|C^{\frac{n+1}{2}}\left(\mathcal{I}_{2\left(\left\lfloor\frac{n+1}{4}\right\rfloor+1\right)}\right)\right|
\geq\left|C^{\frac{n+1}{2}}\left(A_{2\left(\left\lfloor\frac{n+1}{4}\right\rfloor+1\right)}\right)\right|$.
Hence we see that, $2^{n-1}-\frac{n+1}{2}\geq2^{n-1}-\left(\left\lfloor\frac{n+1}{4}\right\rfloor+1\right)$ and this will imply that $n\leq3$, a contradiction.
\end{proof}
Before we proceed to our main result, let us prove an inequality which will be useful in proving Theorem \ref{newboundhypercube}.
\begin{lem}\label{r>3s}
For $n\geq9,$ and 
%
$3\leq q\leq \left\lceil\frac{n}{2}\right\rceil-2$, let $r=\sum\limits_{i=0}^{q}{\binom{n}{i}}$ and $s=\binom{q+\left\lfloor\frac{n}{2}\right\rfloor}{q}$, then $r\geq3s$.
\end{lem}
\begin{proof}
Let us prove by induction on $q$.
For $q=3$, we have to prove that $1+n+\frac{n(n-1)}{2}+\frac{n(n-1)(n-2)}{6}\geq\frac{3\left(\left\lfloor\frac{n}{2}\right\rfloor+3\right)\left(\left\lfloor\frac{n}{2}\right\rfloor+2\right)\left(\left\lfloor\frac{n}{2}\right\rfloor+1\right)}{6}$.	
For $n=9$, $r=130$ which is greater than $3s=105$.
So, let us assume that $n\geq10$.

\noindent On contrary, suppose

\noindent$\begin{array}{rrl}
&1+n+\frac{n(n-1)}{2}+\frac{n(n-1)(n-2)}{6}&<\frac{3\left(\left\lfloor\frac{n}{2}\right\rfloor+3\right)\left(\left\lfloor\frac{n}{2}\right\rfloor+2\right)\left(\left\lfloor\frac{n}{2}\right\rfloor+1\right)}{6}\\
\Longrightarrow&6+5n+n^3&<3\left(\left\lfloor\frac{n}{2}\right\rfloor+3\right)\left(\left\lfloor\frac{n}{2}\right\rfloor+2\right)\left(\left\lfloor\frac{n}{2}\right\rfloor+1\right)\\
&&<3\left(\frac{n}{2}+3\right)\left(\frac{n}{2}+2\right)\left(\frac{n}{2}+1\right)\\
\Longrightarrow&5n^3-36n^2-92n-96&<0\\
\Longrightarrow&5n^3-36n^2-92n-480&<0\\
\Longrightarrow&(5n^2+14n+48)(n-10)&<0\\
\end{array}$

Hence $n<10$, a contradiction. Therefore the result is true for $q=3$.
Now, let us assume that the result is true for $t-1$, that is, for $3\leq t-1< \left\lceil\frac{n}{2}\right\rceil-2$, $\sum\limits_{i=0}^{t-1}{\binom{n}{i}}\geq3\binom{t+\left\lfloor\frac{n}{2}\right\rfloor-1}{t-1}$.
Let us prove the result for $t$,  
$4\leq t\leq\left\lceil\frac{n}{2}\right\rceil-2$. Let us start by considering $r$. 

\noindent$
r=\sum\limits_{i=0}^{t}{\binom{n}{i}}=\sum\limits_{i=0}^{t-1}{\binom{n}{i}}+\binom{n}{t}\geq3\binom{t+\left\lfloor\frac{n}{2}\right\rfloor-1}{t-1}+\binom{n}{t}$

\noindent Claim: $\binom{n}{t}\geq3\binom{t+\left\lfloor\frac{n}{2}\right\rfloor-1}{t}$

\noindent$\begin{array}{rrl}
&\textnormal{Suppose }\binom{n}{t}													&<3\binom{t+\left\lfloor\frac{n}{2}\right\rfloor-1}{t}\\
\Longrightarrow&\frac{n!}{t!(n-t)!}										&<3\left(\frac{(t+\left\lfloor\frac{n}{2}\right\rfloor-1)!}{t!\left(\left\lfloor\frac{n}{2}\right\rfloor-1\right)!}\right)\\
\Longrightarrow&\frac{n(n-1)\cdots(n-t+1)}{t!}				&<\frac{3\left(t+\left\lfloor\frac{n}{2}\right\rfloor-1\right)\left(t+\left\lfloor\frac{n}{2}\right\rfloor-2\right)\cdots\left(t+\left\lfloor\frac{n}{2}\right\rfloor-1-t+4\right)\left(t+\left\lfloor\frac{n}{2}\right\rfloor-1-t+3\right)\left(t+\left\lfloor\frac{n}{2}\right\rfloor-1-t+2\right)\left(t+\left\lfloor\frac{n}{2}\right\rfloor-1-t+1\right)}{t!}\\
\end{array}$

By applying the fact that $t\leq\left\lceil\frac{n}{2}\right\rceil-2$ to all the terms on the right side except the last three terms, we see that 

\noindent$\begin{array}{rrl}
&n(n-1)\cdots(n-t+1)&<3\left(\left\lceil\frac{n}{2}\right\rceil-2+\left\lfloor\frac{n}{2}\right\rfloor-1\right)\left(\left\lceil\frac{n}{2}\right\rceil-2+\left\lfloor\frac{n}{2}\right\rfloor-2\right)\cdots \\ 
 &&\hskip.9cm\cdots\left(\left\lceil\frac{n}{2}\right\rceil-2+\left\lfloor\frac{n}{2}\right\rfloor-t+3\right)
\left(\left\lfloor\frac{n}{2}\right\rfloor+2\right)\left(\left\lfloor\frac{n}{2}\right\rfloor+1\right)\left\lfloor\frac{n}{2}\right\rfloor\\
&
&<3(n-3)(n-4)\cdots(n-t+1)\left(\left\lfloor\frac{n}{2}\right\rfloor+2\right)\left(\left\lfloor\frac{n}{2}\right\rfloor+1\right)\left\lfloor\frac{n}{2}\right\rfloor\\
\Longrightarrow&n(n-1)(n-2)														&<3\left(\left\lfloor\frac{n}{2}\right\rfloor+2\right)\left(\left\lfloor\frac{n}{2}\right\rfloor+1\right)\left\lfloor\frac{n}{2}\right\rfloor\\
								&																			&<\frac{3n}{2}\left(\frac{n}{2}+2\right)\left(\frac{n}{2}+1\right)\\
\Longrightarrow&5n^2-42n-8														&<0\\
\Longrightarrow&5n^2-42n-27														&<0\\
\Longrightarrow&(5n+3)(n-9)														&<0\\
\Longrightarrow&n&<9\textnormal{, a contradiction.}
\end{array}$

Therefore, $\sum\limits_{i=0}^{t}{\binom{n}{i}}
\geq3\binom{t+\left\lfloor\frac{n}{2}\right\rfloor-1}{t-1}+\binom{n}{t}\geq3\binom{t+\left\lfloor\frac{n}{2}\right\rfloor-1}{t-1}+3\binom{t+\left\lfloor\frac{n}{2}\right\rfloor-1}{t}\geq3\binom{t+\left\lfloor\frac{n}{2}\right\rfloor}{t}$.
\end{proof}
Now with the help of Lemma \ref{simplicialcommonngbr}, Corollary \ref{open}, Corollary \ref{i2l}, Lemma \ref{limp2l}, Lemma \ref{parteg}, Lemma \ref{q=1} and Lemma \ref{r>3s}, as mentioned already, we shall establish  an upper bound for the b-chromatic number of some powers of hypercubes which is even better than the one given in Theorem \ref{roughbound}.
\begin{thm}\label{newboundhypercube}
For odd $n\geq5$ and $\frac{n+1}{2}\leq p\leq n-2$, if $r=\sum\limits_{i=0}^{p-\frac{n-1}{2}}{\binom{n}{i}}$ and $s=\binom{p}{p-\frac{n-1}{2}}$, then $b(Q_n^p)\leq2^{n-1}+\left\lfloor\frac{r-s}{2}\right\rfloor$.
For even $n\geq6$ and $\frac{n}{2}+1\leq p\leq n-2$, if $r'=\sum\limits_{i=0}^{p-\frac{n}{2}}{\binom{n}{i}}+ \binom{n-1}{p-\frac{n}{2}}$ and $s'=\binom{p-1}{p-\frac{n}{2}}$, then $b(Q_n^p)\leq2^{n-1}+\left\lfloor\frac{r'-s'}{2}\right\rfloor$.
\end{thm}
\begin{proof}
Let us start with $n$ being odd and $n\geq 5$. By using Lemma \ref{parteg} and Lemma \ref{q=1},  it is enough to assume that $n\geq9$ and $\frac{n+3}{2}\leq p\leq n-2$. 
Suppose $b(Q_n^p)\geq2^{n-1}+\left\lfloor\frac{r-s}{2}\right\rfloor+1$, then by using Lemma \ref{limp2l}, there exists a clique of size $2\left(\left\lfloor\frac{r-s}{2}\right\rfloor+1\right)$, say $A_{2\left(\left\lfloor\frac{r-s}{2}\right\rfloor+1\right)}$ such that $\left|C^p\left(A_{2\left(\left\lfloor\frac{r-s}{2}\right\rfloor+1\right)}\right)\right|\geq2^{n-1}-\left(\left\lfloor\frac{r-s}{2}\right\rfloor+1\right)$.
Since $2\left(\left\lfloor\frac{r-s}{2}\right\rfloor+1\right)\geq r-s+1$, $|C^p(\mathcal{I}_{r-s+1})|\geq\left|C^p\left(\mathcal{I}_{2\left(\left\lfloor\frac{r-s}{2}\right\rfloor+1\right)}\right)\right|$.
By combining these results with Corollary \ref{open} and Corollary \ref{i2l}, we have the following.

\noindent$\begin{array}{rl}
2^{n-1}-(r-2s+1)&>|C^p(\mathcal{I}_{(r-s+1)})|\\
&\geq2^{n-1}-\left(\left\lfloor\frac{r-s}{2}\right\rfloor+1\right)\\
&\geq2^{n-1}-\left(\frac{r-s}{2}+1\right)\\
\textnormal{Therefore}\ r&<3s
\end{array}$

For $p=\frac{n+3}{2}$, $r<3s$ will yield the following.

\noindent$\begin{array}{rrl}
&1+n+\frac{n(n-1)}{2}&<\frac{3\left(\frac{n+3}{2}\right)\left(\frac{n+1}{2}\right)}{2}\\
\Longrightarrow&n^2-8n-1&<0\\
\Longrightarrow&n^2-8n-9&<0\\
\Longrightarrow&(n+1)(n-9)&<0\\
\end{array}$

Thus $n<9$, a contradiction to the assumption that $n\geq9$.
Therefore,  $\frac{n+5}{2}=\frac{n-1}{2}+3\leq p\leq n-2$.

Next, let us consider $n$ to be even with $n\geq6$ and $\frac{n}{2}+1\leq p\leq n-2$. Suppose $b(Q_n^p)\geq2^{n-1}+\left\lfloor\frac{r'-s'}{2}\right\rfloor+1$, then as done in the odd case we can show that $r'<3s'$.
But for $n$ being even and $n\geq6$, if $p=\frac{n}{2}+1$, then $r'=2n$ which is greater than $3s'=\frac{3n}{2}$, a contradiction to $r'<3s'$.
Note that, when $n=6$, there is no $p$ such that $\frac{n}{2}+2\leq p\leq n-2$.
Also, when $n\geq8$ and $p=\frac{n}{2}+2$, $r'<3s'$ implies the following.

\noindent$\begin{array}{rrl}
&1+n+\frac{n(n-1)}{2}+\frac{(n-1)(n-2)}{2}&<\frac{3\left(\frac{n}{2}+1\right)\left(\frac{n}{2}\right)}{2}\\
\Longrightarrow&5n^2-14n+16&<0\\
\Longrightarrow&5n^2-14n+8&<0\\
\Longrightarrow&(5n-4)(n-2)&<0\\
\end{array}$

Thus $n<2$, a contradiction to the assumption that $n\geq8$.
So when $n$ is even, $n\geq 8$ and $\frac{n}{2}+3\leq p\leq n-2$.
Also when $n=8$, there is no $p$ such that $\frac{n}{2}+3\leq p\leq n-2$.
Therefore, for $n$ being even, $n\geq10$ and $\frac{n}{2}+3\leq p\leq n-2$.

Let us introduce a new variable $q$ in the following way. 
Let $q=p-\left\lfloor\frac{n}{2}\right\rfloor$. Then $r, r', s, s'$ can be rewritten as $r=\sum\limits_{i=0}^{q}{\binom{n}{i}}$,  $r'=\sum\limits_{i=0}^{q}{\binom{n}{i}}+ \binom{n-1}{q}$, $s=\binom{q+\left\lfloor\frac{n}{2}\right\rfloor}{q}$ and $s'=\binom{q+\left\lfloor\frac{n}{2}\right\rfloor-1}{q}$. Clearly $r'\geq r$ and $s\geq s'$.
Since $n\geq9$ and $q\geq3$, by using Lemma \ref{r>3s}, we get that $r\geq3s$. This will also imply that  $r'\geq3s'$, a contradiction to both $r<3s$ and $r'<3s'$ in the odd and even case respectively. Hence, we have the desired upper bounds.
\end{proof}
\section{b-coloring of Hamming graphs}
Finally, let us conclude this paper by considering the b-chromatic number of some powers of Hamming graph which is a generalization of Hypercube.
\begin{thm}\label{Hnqp}
   Let $n, q,$ and $ p$ be positive integers. For $2 \leq q\leq n-1$ and $\left\lfloor\frac{n(q-1)}{q}\right \rfloor \leq p\leq n-1$,
$b(H_{n,q}^p) \geq q^{n-1}$.
 \end{thm}
\begin{proof}
Let us consider the subset of the vertices of $\mathbb{Z}_q^n$ given by $\mathcal{S}=\{(a,a,\ldots, a): a\in \mathbb{Z}_q\}$.
Clearly, $\mathcal{S}$ forms a subgroup in $\mathbb{Z}_q^n$ under the addition modulo $q$.
For $x=(x_1,x_2,\ldots, x_n)\in \mathbb{Z}_q^n$, let  $x\mathcal{S}=\{(x_1+a,x_2+a,\ldots,x_n+a): a\in\mathbb{Z}_q\}$ denote the coset of $\mathcal{S}$ with respect to $x$. Note that addition between two elements of $\mathbb{Z}_q$ will always mean addition modulo $q$. 
There are $|\mathbb{Z}_q^n|/|\mathcal{S}|=q^n/q=q^{n-1}$ cosets of $\mathcal{S}$, say $\mathcal{S}_1, \mathcal{S}_2, \ldots, \mathcal{S}_{q^{n-1}}$.
For $1\leq t\leq q^{n-1}$, let us define a coloring $c$ for $H_{n,q}^p$ by $c(v)=t$ if $v\in \mathcal{S}_t$.
Since the distance between any two vertices in $\mathcal{S}$ is $n$ in $H_{n,q}$ and $p\leq n-1$, $\mathcal{S}$ forms an independent set in $H_{n,q}^p$ and hence the coloring $c$ 
is proper.
So, the only thing that remain to show is that every color class of $c$ has a color dominating vertices.
Let us start by considering any two distinct color classes in $H_{n,q}^p$, say $\mathcal{S}_1$ and $\mathcal{S}_2$.
Let $x=(x_1,x_2,\ldots, x_n)\in\mathcal{S}_1$ and $y=(y_1,y_2,\ldots,y_n)\in\mathcal{S}_2$.

\noindent \textbf{Claim}: $x$ is adjacent to some vertex in $\mathcal{S}_2=y\mathcal{S}$.

Suppose $x$ is non-adjacent to any vertex in $y\mathcal{S}$, for each $a\in\mathbb{Z}_q$ there exists some $k_a\geq p+1$ such that $x_i\neq (y_i+a)(\mathrm{mod}~ q)$ for $k_a$ values of $i$, where $i\in [n]$. 
Note that the $k_a$ number of $i's$ need not be the same for two different elements in $\mathbb{Z}_q$.
For $1\leq i\leq n$ and $a\in \mathbb{Z}_q$, let $za=(x_1-(y_1+a),x_2-(y_2+a),\ldots,x_n-(y_n+a))$.
Clearly by using our assumption, for each $a\in \mathbb{Z}_q$, $k_a$ will be the number of non-zero entries in $za$, $k_a\geq p+1$ and $n-k_a\leq n-(p+1)$.
Note that the value $0$ occurs in at most $n-(p+1)$ co-ordinates of a vertex  in $z\mathcal{S}$. Hence the number of co-ordinates having $0$ in $z\mathcal{S}$ is at most $q(n-(p+1))$.
Also one can observe that, for every $i\in[n]$, since the inverse of an element is unique, there exists a unique $b\in \mathbb{Z}_q$ such that $x_i= (y_i+b)(\mathrm{mod}~ q)$. Thus for each co-ordinate the value $0$ occurs exactly once for some vertex $zb$ in $z\mathcal{S}$ and hence the number of co-ordinates having $0$ in $z\mathcal{S}$ is $n$. Therefore
$$n\leq(n-(p+1))q  \Longrightarrow pq+q\leq n(q-1)$$
$$\Longrightarrow p\leq \frac{n(q-1)}{q}-1 \Longrightarrow p< \left\lfloor\frac{n(q-1)}{q}\right\rfloor, \mathrm{~ a ~contradiction}.$$
Thus every vertex in every color class becomes a color dominating vertex.
Hence the coloring $c$ defined by the cosets of $\mathcal{S}$ is a b-coloring using $q^{n-1}$ colors and $b(H_{n,q}^p)\geq q^{n-1}$.
\end{proof}

Corollary \ref{Hnqn-1} is an immediate consequence of Theorem \ref{Hnqp}.
\begin{cor}\label{Hnqn-1}
  Let $q, n, k$ be positive integers and $q\geq2$. Then $b(H_{n,q}^{n-1}) \geq q^{n-1}$.
 \end{cor}
 \begin{proof}
 Let us define $c$ exactly in the same way as we have done in Theorem \ref{Hnqp}. This $c$ will be proper.  To show that every color class of $c$ has a color dominating vertices, let us consider any two distinct color classes in $H_{n,q}^p$, say $\mathcal{S}_1$ and $\mathcal{S}_2$ and $x=(x_1,x_2,\ldots, x_n)\in\mathcal{S}_1$ and $y=(y_1,y_2,\ldots,y_n)\in\mathcal{S}_2=y\mathcal{S}_2$.  As observed in Theorem \ref{Hnqp}, for every $i\in[n]$, since the inverse of an element is unique, there exists a unique $b\in \mathbb{Z}_q$ such that $x_i= (y_i+b)(\mathrm{mod}~ q)$ and hence $x$ and $yb$ differ in at most $n-1$ coordinates. Since the power of the hamming graph is $n-1$, the vertex $x$ will be adjacent to the  vertex $yb$ in $\mathcal{S}_2$. Thus every vertex in every color class becomes a color dominating vertex.
Hence the coloring $c$ defined by the cosets of $\mathcal{S}$ is a b-coloring using $q^{n-1}$ colors and $b(H_{n,q}^{n-1})\geq q^{n-1}$.
 \end{proof}

\subsection*{Acknowledgment}
\small For the first author, this research was supported by Post Doctoral Fellowship, Indian Institute of Technology, Palakkad. And for the second author, this research was supported by SERB DST, Government of India, File no: EMR/2016/007339.
Also, for the third author, this research was supported by the UGC-Basic Scientific Research, Government of India, Student id: gokulnath.res@pondiuni.edu.in.
\bibliographystyle{ams}
\bibliography{ref}  
\end{titlepage}
\end{document}